\documentclass[11pt, a4paper]{article}
\usepackage{amsmath, amsthm, amssymb, url}
\usepackage[margin=1.1in]{geometry}
\usepackage[english]{babel}
\usepackage{authblk}

\usepackage{tikz}
\usetikzlibrary{arrows.meta, calc, patterns}
\usepackage[hidelinks]{hyperref}

\theoremstyle{plain}
\newtheorem{theorem}{Theorem}[section]
\newtheorem{lemma}[theorem]{Lemma}
\newtheorem{corollary}[theorem]{Corollary}
\newtheorem{proposition}[theorem]{Proposition}
\theoremstyle{definition}
\newtheorem{definition}[theorem]{Definition}
\newtheorem{example}[theorem]{Example}
\theoremstyle{remark}
\newtheorem{remark}[theorem]{Remark}

\title{Idempotence criteria for lazy cellular automata}
\author[1]{Edgar Alcal\'{a}-Arroyo}
\author[1]{Alonso Castillo-Ramirez\footnote{\texttt{alonso.castillor@academicos.udg.mx}}}
\author[1]{Maria G. Maga\~na-Ch\'avez}
\author[1]{Luguis de los Santos Ba\~nos}

\affil[1]{Universidad de Guadalajara, Centro Universitario de Ciencias Exactas e Ingenier\'{i}as, Departamento de Matem\'{a}ticas. Guadalajara, M\'{e}xico}

\date{}

\begin{document}

\maketitle

\begin{abstract}
A lazy cellular automaton $\tau:A^G\to A^G$ is determined by a finite neighborhood $S\subseteq G$ containing the group identity $e$, a pattern $p\in A^S$, and a writing symbol $a\in A\setminus\{p(e)\}$: its local rule is equal to the projection to $e$ except when it reads the pattern $p$, in which case it writes $a$. For a nonempty subset $T$ of the fiber $S_a:= p^{-1}(a)$, we introduce compatibility conditions on $p$ along the right translates $St$, $t\in T$. We show that, over every group, the existence of a compatible set is necessary for $\tau$ to be non-idempotent, and that it is also sufficient whenever each element of $S_a$ normalizes the fiber $S_{p(e)}$. In particular, compatibility characterizes non-idempotence over abelian groups, which yields a test using $O(|S|^3)$ elementary operations, instead of an exhaustive search that is exponential in $|S|$. Finally, for every dihedral group $D_n$ with $n\geq 4$, we construct an idempotent lazy cellular automaton that admits a compatible set, showing that the normalization hypothesis cannot simply be dropped. \\


\end{abstract}

\section{Introduction}

Cellular automata (CA) are discrete models defined by a local rule applied simultaneously across the whole space. Following \cite{CSC10}, we consider CA over the configuration space $A^G$ of all functions $G\to A$, where $G$ is a group (the \emph{universe}) and $A$ is a finite set with at least two elements (the \emph{alphabet}). A cellular automaton $\tau:A^G\to A^G$ is determined by a finite \emph{neighborhood} $S\subseteq G$ and a \emph{local rule} $\mu:A^S\to A$ (see Section~\ref{sec:prelim}). A cellular automaton is \emph{lazy} if, for some neighborhood $S$ containing the group identity $e \in G$, its local rule acts as the projection to $e$ except when it reads a single pattern $p \in A^S$, in which case it writes a symbol $a\neq p(e)$. 

Idempotence is a fundamental concept in algebra and computer science, as it ensures strong stability of a system. In our context, a cellular automaton $\tau:A^G\to A^G$ is \emph{idempotent} if $\tau^2:= \tau \circ \tau =\tau$. The characterization of idempotence for lazy CA was studied in \cite{IDEM} and \cite{One-dim}: the first paper presented results for constant, symmetric and quasi-constant patterns over arbitrary groups, while the second addressed one-dimensional cellular automata whose neighborhood $S$ is an interval. In \cite{Lazy}, general results about the order of lazy CA were presented. 

In this paper, we study criteria for the idempotence and non-idempotence of lazy CA. Our starting point is Lemma~\ref{lem:q}, a finite criterion for non-idempotence over arbitrary groups. It reformulates \cite[Lemma~3]{IDEM}, adapted to our shift convention, in terms of patterns on the finite set $SS$ of all products $st$ with $s,t\in S$. This criterion involves the local rule $\mu$ and a search over $|A|^{|SS|}$ patterns. The aim of this article is to replace this search, under hypotheses that hold in particular when $G$ is abelian, by a test on certain subsets of the fiber $S_a := \{s \in S : p(s) = a\}$, which we call \emph{compatible}. Under the same hypotheses, this yields a test for idempotence that uses $O(|S|^3)$ elementary operations.

For a nonempty set $T\subseteq S_a$, we introduce two compatibility conditions on $p$ along the right translates $St$, $t\in T$ (Definition~\ref{def:compatible}). They say that the translated copies of $p$ on these sets fit together with one another and with the values of $p$ on $S\setminus T$ (Lemma~\ref{lem:gluing}). Our main results are the following.
\begin{enumerate}
    \item Over every group, if $\tau$ is not idempotent, then some nonempty $T\subseteq S_a$ is compatible (Proposition~\ref{prop:necessity}).
    \item The converse holds whenever every element of the fiber $S_a$ normalizes the fiber $S_{p(e)}$, i.e. $tS_{p(e)} = S_{p(e)}t$, for all $t \in S_a$ (Theorem~\ref{thm:main}). In particular, compatibility characterizes non-idempotence over every abelian group (Corollary~\ref{cor:abelian}).

\item Over every group, the existence of a compatible set can be decided with $O(|S_a|^2|S|)$ elementary operations. Under the normalization hypothesis of Theorem \ref{thm:main}, this decides idempotence, instead of the exhaustive search over the $|A|^{|SS|}$ patterns on SS (Corollary~\ref{cor:poly}).

    \item The normalization hypothesis of Theorem~\ref{thm:main} cannot simply be dropped: for every $n\geq 4$ we construct an idempotent lazy cellular automaton over the dihedral group $D_n$ that admits a compatible set (Proposition~\ref{prop:dihedral-family}).
\end{enumerate}
Corollary~\ref{cor:abelian} resolves the abelian-group case of the problem, posed in \cite{IDEM}, of characterizing idempotence of lazy cellular automata. It also extends the characterization of \cite{One-dim} for interval neighborhoods in $\mathbb{Z}$ to arbitrary finite neighborhoods in arbitrary abelian groups. 

The structure of the paper is as follows. In Section~\ref{sec:prelim} we fix our
conventions on cellular automata over groups and recall the definition of lazy cellular automata. In Section~\ref{sec:finite} we reformulate non-idempotence in terms of patterns $q \in A^{SS}$ and of the positions of $S$ at which the pattern $p$ is read. In Section~\ref{sec:main} we introduce compatible sets, prove our main results, and show that the existence of a compatible set can be decided with polynomially many elementary operations. In Section~\ref{sec:dihedral-counterexamples} we construct the family of idempotent lazy cellular automata over dihedral groups that admit a compatible set, showing that the normalization hypothesis of Theorem~\ref{thm:main} cannot simply be dropped. Finally, in Section~\ref{sec:conclusions} we discuss the general case and state some open questions.

\section{Preliminaries}\label{sec:prelim}

In this section we recall some basic results from the theory of cellular automata over groups (see \cite[Ch. 1]{CSC10}). Let $G$ be a group with identity element $e$, let $A$ be a finite alphabet with $|A|\geq 2$, and let $S\subseteq G$ be a finite subset of $G$. The elements of $A^S$ are called \emph{patterns on $S$}.

\begin{definition}\label{d.shift}
The \emph{shift action} of $G$ on $A^G$ is the function $\cdot : G \times A^G \to A^G$ defined by
\[ (g \cdot x)(h) := x(hg), \quad \forall x \in A^G,\ g, h \in G. \]
\end{definition}

Note that our shift action is defined via right multiplication by $g$, whereas \cite{CSC10} uses left multiplication by $g^{-1}$. We also define the shift action on patterns. 

\begin{definition}\label{d.shift-patterns}
For a finite $S\subseteq G$, $p\in A^S$ and $g\in G$, define $g\cdot p\in A^{Sg^{-1}}$ by
\[
(g\cdot p)(h):=p(hg), \qquad h\in Sg^{-1}.
\]
Two patterns $p \in A^S$ and $p' \in A^{S'}$ \emph{agree} if $p|_{S\cap S'}=p'|_{S\cap S'}$.
\end{definition}

The shift action on patterns is compatible with restriction in the sense that $(g\cdot x)|_{Sg^{-1}}=g\cdot(x|_S)$ for every $x\in A^G$. A family of pairwise agreeing patterns $p_i\in A^{S_i}$ \emph{glues} to a unique pattern on $\bigcup_i S_i$ whose restriction to each $S_i$ is $p_i$.

\begin{definition}
A \emph{cellular automaton} is a function $\tau : A^G \to A^G$ such that there exist a finite subset $S \subseteq G$, called a \emph{neighborhood} of $\tau$, and a \emph{local function}, or \emph{local rule}, $\mu : A^S \to A$, such that
\[ \tau(x)(g) = \mu( (g \cdot x) \vert_S ), \quad \forall x \in A^G,\ g \in G. \]
\end{definition}

\begin{definition}
A cellular automaton $\tau:A^G\to A^G$ is \emph{lazy} if it admits a local defining rule $\mu:A^S\to A$ with $e\in S$ for which there exists $p\in A^S$ such that
\[
\mu(z)=z(e) \Longleftrightarrow z\neq p, \qquad \forall z\in A^S.
\]
We call $p$ the \emph{unique active transition} of $\tau$, and $a:=\mu(p)\in A\setminus\{p(e)\}$ the \emph{writing symbol} of $\tau$.
\end{definition}

Throughout the paper, $\tau : A^G \to A^G$ denotes a lazy cellular automaton with local rule $\mu:A^S\to A$, unique active transition $p\in A^S$ and writing symbol $a$.

\begin{example}
Let $G=\mathbb{Z}$, $A := \{ 0,1 \}$ and $S := \{-1,0,1\}$. The cellular automata that admit the neighborhood $S$ are called \emph{elementary cellular automata} (ECA), and they are labeled by a \emph{Wolfram number} (see \cite[Sec.~2.5]{Kari}). The ECA rule 76 is lazy, because its local rule $\mu: A^S\to A$ is given by the following table:
\[ \begin{tabular}{c|cccccccc}
$z \in A^S$ & $111$ & $110$ & $101$ & $100$ & $011$ & $010$ & $001$ & $000$ \\ \hline
$\mu(z) \in A$ & $0$ & $1$ & $0$ & $0$ & $1$ & $1$ & $0$ & $0$
\end{tabular}\]
In this example, the pattern $p = 111 \in A^S$ is the unique active transition, and the writing symbol is $a=0$.
\end{example}

The \emph{fiber} of $p$ at a symbol $b \in A$ is
\[
S_b=\{s\in S:p(s)=b\}.
\]
Note that, when $b=a$ is the writing symbol, $e\notin S_a$ because $p(e)\neq a$.

We say that a pattern $p\in A^S$ \emph{appears} in $x \in A^G$ if there is $g \in G$ such that $(g \cdot x)|_S = p$.

\begin{lemma}\label{cor-idem}
The lazy cellular automaton $\tau$ is not idempotent if and only if there exists $x \in A^G$ such that $p$ appears in $\tau(x)$.
\end{lemma}

Lemma~\ref{cor-idem} follows from \cite[Lemma~2]{IDEM}. Intuitively, $\tau$ is non-idempotent if there exists a configuration $x$ such that the pattern $p$ reappears after applying $\tau$. Applying $\tau$ again then forces a change, so $\tau(x)\neq\tau^2(x)$.

\section{A finite criterion for non-idempotence}\label{sec:finite}

Let $SS:=\{st:s,t\in S\}$. Since $e\in S$, we have $S\subseteq SS$ and $St\subseteq SS$ for every $t\in S$. For $q\in A^{SS}$ and $s\in S$, define $q_s\in A^S$ by
\[
q_s(u):=q(us), \qquad \forall u\in S.
\]
In terms of Definition~\ref{d.shift-patterns}, $q_s=(s\cdot q)|_S$, which is defined because $Ss\subseteq SS$. If $x\in A^G$ and $q=x|_{SS}$, then $q_s=(s\cdot x)|_S$. Hence the patterns that $\tau$ reads at the cells $s\in S$ are completely determined by $x|_{SS}$.

The following lemma is a reformulation of \cite[Lemma~3]{IDEM}, adapted to the shift convention of Definition~\ref{d.shift} and stated in terms of patterns on $SS$. The explicit search over finitely many patterns also appears in \cite[Remark~1]{One-dim}. Since $A$ and $SS$ are finite, the lemma yields a finite combinatorial test for idempotence. The algorithmic statements in this paper assume that products and equality of the relevant group's elements can be computed.

\begin{lemma}\label{lem:q}
The lazy cellular automaton $\tau$ is not idempotent if and only if there exists $q\in A^{SS}$ such that
\begin{equation}\label{eq:criterio-mu}
\mu(q_s)=p(s)
\qquad \forall s\in S.
\end{equation}
\end{lemma}

\begin{proof}
Suppose that $\tau$ is not idempotent. By Lemma~\ref{cor-idem}, there is $x\in A^G$ such that $p$ appears in $\tau(x)$; that is, there is $g\in G$ such that $(g\cdot\tau(x))|_S=p$. Define $q\in A^{SS}$ by
\[
q(z):=x(zg)=(g\cdot x)(z),
\qquad z\in SS.
\]
For $s,u\in S$ we have
\[
q_s(u)=q(us)=x(usg)=((sg)\cdot x)(u),
\]
and hence $q_s=((sg)\cdot x)|_S$. Consequently,
\[
\mu(q_s)=\mu(((sg)\cdot x)|_S)=\tau(x)(sg)=(g\cdot\tau(x))(s)=p(s)
\]
for every $s\in S$.

Conversely, suppose that there exists $q\in A^{SS}$ such that $\mu(q_s)=p(s)$ for all $s\in S$. Extend $q$ arbitrarily to a configuration $x\in A^G$ with $x|_{SS}=q$. Then
\[
\tau(x)(s)=\mu((s \cdot x)|_S)=\mu(q_s)=p(s)
\qquad \forall s\in S,
\]
and therefore $\tau(x)|_S=p$. In particular,
\[
\tau(x)(e)=p(e)
\qquad\text{and}\qquad
\tau^2(x)(e)=\mu(\tau(x)|_S)=\mu(p)=a.
\]
Since $a\neq p(e)$, we conclude that $\tau^2(x)\neq\tau(x)$.
\end{proof}

\begin{example}\label{ex.1}
Let $G=\mathbb{Z}$, $S = \{-1,0,1,2\}$ and $A=\{0,1\}$, and consider the lazy cellular automaton with unique active transition
\[ p =\left( \begin{tabular}{cccc}
  $-1$ & $0$ & $1$ & $2$ \\
 0 & 0 & 1 & 1
\end{tabular} \right)
\]
and writing symbol $a=1$. Here we are using the two-line convention to denote the pattern $p \in A^S$, where the first row corresponds to the elements of $S$ and the second row to the image of these elements under $p$. This lazy cellular automaton is not idempotent, because
\[ q =\left( \begin{tabular}{ccccccc}
 $-2$ &  $-1$ & $0$ & $1$ & $2$ & $3$ & $4$ \\
  0 & 0 & 0 & 0 & 1 & 1 & 0
\end{tabular} \right) \in A^{SS}
\]
satisfies
\[ \mu(q_{-1})\, \mu(q_0)\, \mu(q_1)\, \mu(q_2) = \mu(0000)\, \mu(0001)\, \mu(0011)\, \mu(0110) = 0011 = p.  \]
\end{example}

For $q\in A^{SS}$, define the set of \emph{active centers} within $S$ by
\[
T(q):=\{t\in S:q_t=p\}.
\]

\begin{lemma}\label{lem:Tq-nonempty}
If $q\in A^{SS}$ satisfies \eqref{eq:criterio-mu}, then $\emptyset \neq T(q) \subseteq S_a$.
\end{lemma}

\begin{proof}
If $t \in T(q)$, then $q_t = p$, so \eqref{eq:criterio-mu} gives
\[ p(t)= \mu(q_t) = \mu(p) = a. \]
Hence, $T(q) \subseteq S_a$. Suppose, for a contradiction, that $T(q)=\emptyset$. Then $q_s\neq p$ for every $s\in S$, so laziness and \eqref{eq:criterio-mu} give
\[
p(s)=\mu(q_s)=q_s(e)=q(s), \qquad \forall s\in S.
\]
Since $e\in S$, it follows that $q_e(u)=q(u)=p(u)$ for every $u\in S$; that is, $q_e=p$. Hence $e\in T(q)$, a contradiction.
\end{proof}

\begin{lemma}\label{thm:q-T}
The lazy cellular automaton $\tau$ is not idempotent if and only if there exist $q\in A^{SS}$ and a nonempty set $T\subseteq S_a$ such that
\begin{align}
q_t&=p &&\text{for all }t\in T,\label{eq:copias-T}\\
q(s)&=p(s) &&\text{for all }s\in S\setminus T,\label{eq:base-fuera-T}\\
q_s&\neq p &&\text{for all }s\in S\setminus T.\label{eq:no-copias-fuera-T}
\end{align}
In this case, $T=T(q)$.
\end{lemma}

\begin{proof}
First suppose that $\tau$ is not idempotent. By Lemma~\ref{lem:q}, there exists $q\in A^{SS}$ satisfying \eqref{eq:criterio-mu}. Let $T:=T(q)$. By Lemma~\ref{lem:Tq-nonempty}, $\emptyset \neq T \subseteq S_a$. Moreover, if $s\in S\setminus T$, then $q_s\neq p$ and
\[
p(s)=\mu(q_s)=q_s(e)=q(s).
\]
Therefore, conditions \eqref{eq:base-fuera-T} and \eqref{eq:no-copias-fuera-T} are satisfied.

Now suppose that $q\in A^{SS}$ and a nonempty set $T\subseteq S_a$ satisfy \eqref{eq:copias-T}--\eqref{eq:no-copias-fuera-T}. For $t\in T$,
\[
\mu(q_t)=\mu(p)=a=p(t),
\]
while, for $s\in S\setminus T$, conditions \eqref{eq:no-copias-fuera-T} and \eqref{eq:base-fuera-T} give
\[
\mu(q_s)=q_s(e)=q(s)=p(s).
\]
Therefore, $\mu(q_s)=p(s)$ for all $s\in S$, and Lemma~\ref{lem:q} implies that $\tau$ is not idempotent. Finally, \eqref{eq:copias-T} and \eqref{eq:no-copias-fuera-T} say precisely that $T=T(q)$.
\end{proof}

\begin{corollary}\label{cor:empty-fiber}
If $S_a=\emptyset$, that is, if the writing symbol $a$ does not occur in $p$, then $\tau$ is idempotent. In particular, every lazy cellular automaton whose unique active transition is a constant pattern is idempotent.
\end{corollary}

\begin{example}\label{ex:Z10-q}
Let $G=\mathbb{Z}_{10}$, $S = \{0,1,3,4\}$ and $A=\{0,1\}$, and let $\tau$ be the lazy cellular automaton with unique active transition $p=(p(0),p(1),p(3),p(4))=(0,1,1,1)$ and writing symbol $a=1$, so that $S_1=\{1,3,4\}$. In additive notation, $S+S=\{0,1,\dots,8\}$. The pattern $q\in A^{S+S}$ with $(q(0),\dots,q(8))=(0,0,1,0,1,1,1,1,0)$ satisfies $q_1=q_3=p$, $q_0=(0,0,0,1)\neq p$, $q_4=(1,1,1,0)\neq p$, $q(0)=p(0)$ and $q(4)=p(4)$. By Lemma~\ref{thm:q-T} with $T=\{1,3\}\subseteq S_1$, the automaton $\tau$ is not idempotent.
\end{example}

\section{Compatible sets and the main theorem}\label{sec:main}

We now replace the search for $q$ in Lemma~\ref{thm:q-T} by a search over subsets $T\subseteq S_a$ satisfying two compatibility conditions on $p$. For $t\in S$, the pattern $t^{-1}\cdot p\in A^{St}$, given by
\[
(t^{-1}\cdot p)(z)=p(zt^{-1}), \qquad z\in St,
\]
is the copy of $p$ placed on the right translate $St=\{ut:u\in S\}$. Since $q_t(u)=q(ut)$ and $(t^{-1}\cdot p)(ut)=p(u)$ for $u\in S$, we have
\begin{equation}\label{eq:copy}
q_t=p \quad\Longleftrightarrow\quad q|_{St}=t^{-1}\cdot p.
\end{equation}
This motivates the following definition.

\begin{definition}\label{def:compatible}
A nonempty set $T\subseteq S_a$ is \emph{compatible} (with $p$) if it satisfies the following two conditions:
\begin{enumerate}
    \item For all $t,r\in T$ and all $z\in(St)\cap(Sr)$,
    \begin{equation}\label{eq:C1}
        p(zt^{-1})=p(zr^{-1}).
    \end{equation}
    \item For all $t\in T$ and all $z\in(St)\cap(S\setminus T)$,
    \begin{equation}\label{eq:C2}
        p(zt^{-1})=p(z).
    \end{equation}
\end{enumerate}
\end{definition}

Condition \eqref{eq:C1} says that the copies of $p$ placed on the translates $St$, $t\in T$, agree wherever they overlap. Condition \eqref{eq:C2} says that each of these copies agrees with $p$ itself on the positions of $S$ outside $T$.

\begin{lemma}\label{lem:gluing}
A nonempty set $T\subseteq S_a$ is compatible if and only if the patterns
\[
t^{-1}\cdot p\in A^{St}\ \ (t\in T)
\qquad\text{and}\qquad
p|_{S\setminus T}\in A^{S\setminus T}
\]
pairwise agree. In this case they glue to a pattern $P_T$ on $D_T:=\bigl(\bigcup_{t\in T}St\bigr)\cup(S\setminus T)\subseteq SS$.
\end{lemma}

\begin{proof}
For $t,r\in T$ and $z\in St\cap Sr$, we have $(t^{-1}\cdot p)(z)=p(zt^{-1})$ and $(r^{-1}\cdot p)(z)=p(zr^{-1})$. Hence $t^{-1}\cdot p$ and $r^{-1}\cdot p$ agree if and only if \eqref{eq:C1} holds for the pair $t,r$. Similarly, for $t\in T$ and $z\in St\cap(S\setminus T)$, we have $(t^{-1}\cdot p)(z)=p(zt^{-1})$ and $p|_{S\setminus T}(z)=p(z)$. Hence $t^{-1}\cdot p$ and $p|_{S\setminus T}$ agree if and only if \eqref{eq:C2} holds for $t$. Finally, $D_T\subseteq SS$ because $St\subseteq SS$ and $S\subseteq SS$.
\end{proof}

Thus $T$ is compatible precisely when the copies of $p$ placed on the translates $St$, $t\in T$, fit together with one another and with the values of $p$ on $S\setminus T$.

\begin{proposition}\label{prop:necessity}
Let $G$ be any group. If $\tau$ is not idempotent, then there exists a compatible set $T\subseteq S_a$.
\end{proposition}
\begin{proof}
By Lemma~\ref{thm:q-T}, there exist $q\in A^{SS}$ and a nonempty set $T\subseteq S_a$ satisfying \eqref{eq:copias-T} and \eqref{eq:base-fuera-T}. By \eqref{eq:copias-T} and \eqref{eq:copy}, $q|_{St}=t^{-1}\cdot p$ for every $t\in T$, and \eqref{eq:base-fuera-T} says that $q|_{S\setminus T}=p|_{S\setminus T}$. These patterns are restrictions of the single pattern $q$, so they pairwise agree, and $T$ is compatible by Lemma~\ref{lem:gluing}.
\end{proof}

\begin{theorem}\label{thm:main}
Let $G$ be a group, and suppose that
\begin{equation}\label{eq:H}
tS_{p(e)} =  S_{p(e)}t \qquad \text{for all } t\in S_a.
\end{equation}
Then $\tau$ is not idempotent if and only if there exists a compatible set $T\subseteq S_a$.
\end{theorem}
\begin{proof}
The direct implication is Proposition~\ref{prop:necessity}. Conversely, let $T\subseteq S_a$ be compatible, let $P_T$ be the pattern on $D_T$ given by Lemma~\ref{lem:gluing}, and let $q\in A^{SS}$ be any extension of $P_T$. Then $q|_{St}=t^{-1}\cdot p$, that is, $q_t=p$ by \eqref{eq:copy}, for every $t\in T$, and $q|_{S\setminus T}=p|_{S\setminus T}$. Hence conditions \eqref{eq:copias-T} and \eqref{eq:base-fuera-T} of Lemma~\ref{thm:q-T} hold.

It remains to prove that $q_s\neq p$ for every $s\in S\setminus T$. Let $s\in S\setminus T$. If $p(s)\neq p(e)$, then
\[
q_s(e)=q(s)=p(s)\neq p(e),
\]
and hence $q_s\neq p$. Suppose now that $p(s) = p(e)$, and fix any $t \in T$, which is possible because $T \neq \emptyset$. Since $t \in S_a$, hypothesis \eqref{eq:H} gives $v \in S_{p(e)}$ such that $ts = vt$. Since $q_t = p$,
\[ q_s(t)=q(ts)=q(vt)= q_t(v) = p(v) = p(e) \neq a = p(t).\]
Hence $q_s\neq p$.

Consequently, all the hypotheses of Lemma~\ref{thm:q-T} are satisfied, and $\tau$ is not idempotent.
\end{proof}

Hypothesis \eqref{eq:H} is used only once in the proof: to write $ts$ as $vt$ with $v \in S_{p(e)}$, so that the copy of $p$ on $St$ forces $q(ts) = p(e)$. Since $te = et$, hypothesis \eqref{eq:H} is equivalent to $t(S_{p(e)} \setminus \{e \})t^{-1} = (S_{p(e)} \setminus \{e\})$ for all $t \in S_a$.

\begin{corollary}\label{cor:abelian}
Suppose that one of the following holds:
\begin{enumerate}
    \item $G$ is abelian;
    \item $S_a$ is contained in the center of $G$;
    \item $p(s)\neq p(e)$ for every $s\in S\setminus\{e\}$.
\end{enumerate}
Then $\tau$ is not idempotent if and only if there exists a compatible set $T\subseteq S_a$. 
\end{corollary}

\begin{proof}
In each case hypothesis \eqref{eq:H} holds, so the result follows from Theorem~\ref{thm:main}.
\end{proof}

We now discuss the computational cost of deciding idempotence. Throughout, we use a \emph{finite-set lookup model}: the products $us$ with $u, s \in S$ are computed once, the elements of $SS$ are represented by labels, and the following operations take constant time: evaluating a pattern at an element of its domain, testing membership in a subset of $SS$, and recovering $u \in S$ from $z = ur$ when $z \in Sr$. Costs are measured in the number of such lookups; they are not unconditional running-time bounds for arbitrarily presented groups.

In this model, the test given by Lemma~\ref{lem:q} examines the $|A|^{|SS|}$ patterns $q \in A^{SS}$ and, for each of them, reads the $|S|$ values
of $q_s$ for every $s \in S$ in order to verify \eqref{eq:criterio-mu}. Hence it uses $O\bigl(|A|^{|SS|}\,|S|^2\bigr)$ lookups. Since $|S| \leq |SS| \leq |S|^2$, this is exponential in $|S|$ for every neighborhood; for instance, $|SS| = 2|S| - 1$ when $S$ is an interval of $\mathbb{Z}$. We now show that the existence of a compatible set can be decided with a number of lookups that is polynomial in $|S|$ and independent of $|A|$.

For $z \in S$, define
\[   F(z) := \{\, w \in Sz \cap S : p(wz^{-1}) \neq p(w) \,\}. \]	
A set $C \subseteq S$ is \emph{$F$-closed} if $F(z) \subseteq C$ for all $z \in C$.
For $t \in S$, let $C(t)$ be the smallest $F$-closed subset of $S$ containing $t$. This exists because $S$ itself is $F$-closed and the intersection of $F$-closed sets is $F$-closed.

\begin{proposition}\label{prop:closure}
Let $G$ be any group.
\begin{enumerate}
\item A nonempty set $T \subseteq S_a$ satisfies the compatibility condition \eqref{eq:C2} if and only if it is $F$-closed.
\item If $T \subseteq S_a$ is compatible and $t \in T$, then $C(t) \subseteq T$ and $C(t)$ is compatible.
\item A compatible set $T \subseteq S_a$ exists if and only if there is $t \in S_a$ such that
      $C(t) \subseteq S_a$ and $C(t)$ satisfies the compatibility condition \eqref{eq:C1}.
\end{enumerate}
\end{proposition}

\begin{proof}
(1) Condition \eqref{eq:C2} states that, for all $t \in T$ and $z \in St \cap S$,
if $p(zt^{-1}) \neq p(z)$ then $z \in T$; that is, $F(t) \subseteq T$ for all $t \in T$.

(2) By the previous point, $T$ is an $F$-closed set containing $t$, so $C(t) \subseteq T \subseteq S_a$.
Condition \eqref{eq:C1} only involves pairs of elements, so it is inherited by $C(t)$,
and $C(t)$ satisfies \eqref{eq:C2} by the previous point.

(3) If $T$ is compatible, take any $t \in T$ and apply (2). Conversely, such a set $C(t)$
is nonempty, contained in $S_a$, satisfies \eqref{eq:C1}, and satisfies \eqref{eq:C2} by (1).
\end{proof}

\begin{corollary}\label{cor:poly}
Let $G$ be any group, and write $m := |S_a|$ and $n := |S|$.
\begin{enumerate}
\item In the lookup model above, the existence of a compatible set
      $T \subseteq S_a$ can be decided with $O(m^2 n)$ lookups. If no compatible
      set exists, then $\tau$ is idempotent.
\item If hypothesis \eqref{eq:H} of Theorem~\ref{thm:main} holds, then $\tau$ is
      not idempotent if and only if there exists $t \in S_a$ such that
      $C(t) \subseteq S_a$ and $C(t)$ satisfies the compatibility condition
      \eqref{eq:C1}. Hence idempotence can be decided with $O(m^2 n)$ lookups,
      which is at most $O(|S|^3)$.
\end{enumerate}
\end{corollary}

\begin{proof}
\begin{enumerate}
\item By Proposition~\ref{prop:closure}(3), it suffices to decide whether some $t \in S_a$ satisfies $C(t) \subseteq S_a$ and \eqref{eq:C1} holds on $C(t)$. First, compute $F(t)$ for every $t \in S_a$, recording whether
$F(t) \subseteq S_a$; this takes $O(mn)$ lookups. Second, build a table recording, for each pair $t, r \in S_a$, whether \eqref{eq:C1} holds for that pair: for every $u \in S$, test whether $ut \in Sr$ and, if so, compare $p(u)$ with $p(utr^{-1})$. This takes $O(n)$ lookups per pair, hence $O(m^2 n)$ in total. Now fix $t \in S_a$. The set $C(t)$ is obtained by starting with $\{t\}$ and adding $F(z)$ for each new element $z$, stopping if some $F(z)$ is not contained in $S_a$. Since at most $m$ elements are added and each set $F(z)$ used has at most $m$ elements, this takes $O(m^2)$ lookups. Checking \eqref{eq:C1} on $C(t)$ takes $O(m^2)$ lookups in the table. Over the $m$ choices of $t$ this gives $O(m^3)$, and $m^3 \leq m^2 n$. The last assertion is Proposition~\ref{prop:necessity}.

\item The equivalence follows from Theorem~\ref{thm:main} and Proposition~\ref{prop:closure}(3), and the bound from the previous point.
\end{enumerate}
\end{proof}

\begin{example}\label{ex:Z10-T}
For the cellular automaton $\tau : A^{\mathbb{Z}_{10}} \to A^{\mathbb{Z}_{10}}$ of Example~\ref{ex:Z10-q}, the singleton $T=\{1\}$ is compatible. Condition \eqref{eq:C1} is automatic, and $(S+1)\cap(S\setminus T)=\{1,2,4,5\}\cap\{0,3,4\}=\{4\}$, where $p(4-1)=p(3)=1=p(4)$. Hence $\tau$ is not idempotent by Corollary~\ref{cor:abelian}. The set $T=\{1,3\}$ used in Example~\ref{ex:Z10-q} is compatible as well. Here $F(1) = \{1\}$, so $C(1) = \{1\}$.
\end{example}

\begin{example}\label{ex:singletons}
Let $G=\mathbb{Z}$, $A=\{0,1\}$ and $S=\{-3,-2,0,1,3\}$, and consider the lazy cellular automaton with unique active transition
\[ p =\left( \begin{tabular}{ccccc}
  $-3$ & $-2$ & $0$ & $1$ & $3$ \\
 0 & 1 & 0 & 1 & 0
\end{tabular} \right)
\]
and writing symbol $a=1$, so that $S_1=\{-2,1\}$.

The set $T=\{-2,1\}$ is compatible. With $t=-2$ and $r=1$, condition \eqref{eq:C1} requires $p(z+2)=p(z-1)$ for all
\[
z\in(S-2)\cap(S+1)=\{-5,-4,-2,-1,1\}\cap\{-2,-1,1,2,4\}=\{-2,-1,1\}.
\]
This holds, because $p(0)=0=p(-3)$, $p(1)=1=p(-2)$ and $p(3)=0=p(0)$. Since $S\setminus T=\{-3,0,3\}$ is disjoint from both $S-2$ and $S+1$, condition \eqref{eq:C2} holds vacuously. By Corollary~\ref{cor:abelian}, $\tau$ is not idempotent.

However, neither singleton is compatible. For $T=\{-2\}$, the point $z=1$ lies in $(S-2)\cap(S\setminus T)$, and $p(1+2)=p(3)=0\neq 1=p(1)$. For $T=\{1\}$, the point $z=-2$ lies in $(S+1)\cap(S\setminus T)$, and $p(-2-1)=p(-3)=0\neq 1=p(-2)$.

In this example, $F(-2) = F(1) = \{-2, 1\}$, so $C(-2) = C(1) = \{-2, 1\}$.
\end{example}

In Example \ref{ex:Z10-T} there is a compatible singleton, but Example \ref{ex:singletons} shows that the minimal compatible sets $C(t)$ need not be singletons.

\begin{remark}\label{rem:figures}
Figure~\ref{d1} illustrates condition~\eqref{eq:C1} and Lemma~\ref{lem:gluing} for the cellular automaton of Example~\ref{ex:singletons}, with $T=\{-2,1\}$. Over $\mathbb{Z}$ we have $St=S+t$, and the copy of $p$ on $S+t$ is $z\mapsto p(z-t)$. Each point $z$ in the intersection of $S-2$ and $S+1$ receives one value from each copy, and \eqref{eq:C1} requires these values to be equal.

Figure~\ref{d2} illustrates condition~\eqref{eq:C2} for the cellular automaton of Example~\ref{ex:Z10-T}, with $T=\{1\}$. A point $z$ of $S+1$ that also lies in $S\setminus T$ receives the value $p(z-1)$ from the copy and the value $p(z)$ from $p$ itself, and \eqref{eq:C2} requires these values to be equal. At the position $t$ itself, the copy of $p$ on $St$ carries $p(e)$, which differs from $p(t)=a$; this is why \eqref{eq:C2} is imposed only on $S\setminus T$.
\end{remark}

\begin{figure}[ht]
\centering
\begin{tikzpicture}[x=0.65cm, y=0.65cm,
    cell/.style={draw, minimum size=0.65cm, inner sep=0pt, anchor=center},
    lab/.style={anchor=east, font=\small}]
    \foreach \x in {-2,-1,1} {\fill[gray!25] (\x-0.5,2.45) rectangle (\x+0.5,4.55);}
    \foreach \x in {-5,...,4} {\node[font=\scriptsize] at (\x,5.1) {$\x$};}
    \node[lab] at (-5.8,4) {copy of $p$ on $S-2$};
    \foreach \x/\v in {-5/0,-4/1,-2/0,-1/1,1/0} {\node[cell] at (\x,4) {$\v$};}
    \draw[line width=1.6pt] (-2.5,3.5) rectangle (-1.5,4.5);
    \node[lab] at (-5.8,3) {copy of $p$ on $S+1$};
    \foreach \x/\v in {-2/0,-1/1,1/0,2/1,4/0} {\node[cell] at (\x,3) {$\v$};}
    \draw[line width=1.6pt] (0.5,2.5) rectangle (1.5,3.5);
    \node[lab] at (-5.8,1.8) {$p|_{S\setminus T}$};
    \foreach \x/\v in {-3/0,0/0,3/0} {\node[cell] at (\x,1.8) {$\v$};}
    \draw[gray] (-5.6,1.0) -- (4.6,1.0);
    \node[lab] at (-5.8,0.2) {glued pattern $P_T$};
    \foreach \x/\v in {-5/0,-4/1,-3/0,-2/0,-1/1,0/0,1/0,2/1,3/0,4/0} {\node[cell] at (\x,0.2) {$\v$};}
\end{tikzpicture}
\caption{Condition~\eqref{eq:C1} for Example~\ref{ex:singletons}: $G=\mathbb{Z}$, $S=\{-3,-2,0,1,3\}$, $p=(0,1,0,1,0)$ and $T=\{-2,1\}$. Each of the first two rows is a translated copy of $p$; the thick cell marks the position $t$, where the copy places $p(0)$. In the shaded columns the two copies overlap and carry the same values, so together with $p|_{S\setminus T}$ they glue to the pattern $P_T$ of Lemma~\ref{lem:gluing}. Here $S\setminus T$ meets neither copy, so \eqref{eq:C2} holds vacuously.}
\label{d1}
\end{figure}

\begin{figure}[ht]
\centering
\begin{tikzpicture}[x=0.65cm, y=0.65cm,
    cell/.style={draw, minimum size=0.65cm, inner sep=0pt, anchor=center},
    lab/.style={anchor=east, font=\small}]
    \fill[gray!25] (3.5,1.45) rectangle (4.5,3.55);
    \foreach \x in {0,...,5} {\node[font=\scriptsize] at (\x,4.1) {$\x$};}
    \node[lab] at (-0.8,3) {copy of $p$ on $S+1$};
    \foreach \x/\v in {1/0,2/1,4/1,5/1} {\node[cell] at (\x,3) {$\v$};}
    \draw[line width=1.6pt] (0.5,2.5) rectangle (1.5,3.5);
    \node[lab] at (-0.8,2) {$p$ on $S$};
    \foreach \x/\v in {0/0,3/1,4/1} {\node[cell] at (\x,2) {$\v$};}
    \node[cell, pattern=north east lines, pattern color=gray] at (1,2) {};
    \node[fill=white, inner sep=1pt] at (1,2) {$1$};
    \draw[dashed] (0.42,1.42) rectangle (1.58,3.58);
    \node[font=\small] at (1,0.85) {$1\in T$: no condition};
    \draw[gray] (-0.6,0.2) -- (5.6,0.2);
    \node[lab] at (-0.8,-0.6) {glued pattern $P_T$};
    \foreach \x/\v in {0/0,1/0,2/1,3/1,4/1,5/1} {\node[cell] at (\x,-0.6) {$\v$};}
\end{tikzpicture}
\caption{Condition~\eqref{eq:C2} for Example~\ref{ex:Z10-T}: $G=\mathbb{Z}_{10}$, $S=\{0,1,3,4\}$, $p=(0,1,1,1)$ and $T=\{1\}$ (hatched). The copy of $p$ on $S+1$ meets $S\setminus T$ only in the shaded column $4$, where both rows carry the value $1$. At $1\in T$ the copy carries $p(0)=0$, while $p(1)=a=1$; condition \eqref{eq:C2} imposes nothing there, and the glued pattern $P_T$ takes the value of the copy.}
\label{d2}
\end{figure}


\section{A family of counterexamples over dihedral groups}\label{sec:dihedral-counterexamples}

In this section we show that hypothesis \eqref{eq:H} cannot simply be dropped from Theorem~\ref{thm:main}. More precisely, we construct a family of lazy cellular automata over the dihedral groups $D_n$, $n\geq 4$, which admit a compatible set but are idempotent.

For every integer $n\geq 3$, consider the dihedral group
\[
D_n=\langle r,s\mid r^n=s^2=e,\ srs=r^{-1}\rangle
\]
of order $2n$. Define $v:=r^2s$ and take the neighborhood
\[
S:=\{e,s,r,v\}=\{e,s,r,r^2s\},
\]
which has four distinct elements because $n\geq 3$. Let $A=\{0,1\}$, and define the pattern $p\in A^S$ by
\[
\begin{array}{c|cccc}
g    & e & s & r & r^2s\\
\hline
p(g) & 0 & 0 & 1 & 1
\end{array}.
\]
Let $\tau\colon A^{D_n}\to A^{D_n}$ be the lazy cellular automaton with unique active transition $p$ and writing symbol $a:=1$, which is valid because $p(e)=0$. Then
\[
S_a=\{r,v\}.
\]
Hypothesis \eqref{eq:H} fails for $t = r \in S_a$: since $S_{p(e)} = \{e, s\}$ and $rsr^{-1} = r^2s \notin \{e, s\}$, we have $rS_{p(e)} \neq S_{p(e)}r$.

\begin{proposition}\label{prop:dihedral-family}
For every $n\geq 3$, the lazy cellular automaton $\tau$ is idempotent. If $n\geq 4$, the set $T=\{r\}$ is compatible. Consequently, for $n\geq 4$, the equivalence of Theorem~\ref{thm:main} does not hold over $D_n$.
\end{proposition}

\begin{proof}
We first verify that $T=\{r\}$ is compatible when $n\geq 4$. Since $T$ has only one element, condition~\eqref{eq:C1} is automatically satisfied. For condition~\eqref{eq:C2}, observe that $sr=r^{-1}s$ and $vr=r^2sr=rs$, so
\[
Sr=\{r,\ r^{-1}s,\ r^2,\ rs\},
\qquad
S\setminus T=\{e,\ s,\ r^2s\}.
\]
These sets are disjoint. The rotations $r$ and $r^2$ differ from $e$, and they are not reflections. For the reflections, $r^{-1}s=s$ and $rs=s$ would imply $r=e$, $rs=r^2s$ would imply $r=e$, and $r^{-1}s=r^2s$ would imply $r^3=e$, which is impossible for $n\geq 4$. Therefore $Sr\cap(S\setminus T)=\emptyset$, and condition~\eqref{eq:C2} is satisfied vacuously.

We now prove that $\tau$ is idempotent for every $n\geq 3$. Suppose, for a contradiction, that $\tau$ is not idempotent. By Lemma~\ref{lem:q}, there exists $q\in A^{SS}$ such that
\begin{equation}
\label{eq:dihedral-witness}
\mu(q_g)=p(g)
\qquad\text{for every }g\in S.
\end{equation}
By Lemma~\ref{lem:Tq-nonempty}, the set $T(q)$ is nonempty and $T(q)\subseteq S_a=\{r,v\}$. Since $\mu(p)=1$, \eqref{eq:dihedral-witness} gives $q_e\neq p$ and $q_s\neq p$, and then laziness gives
\[
q(e)=q_e(e)=\mu(q_e)=p(e)=0
\qquad\text{and}\qquad
q(s)=q_s(e)=\mu(q_s)=p(s)=0.
\]

Assume first that $v\in T(q)$, so that $q_v=p$. Since $v=r^2s$ is a reflection, $v^2=e$, and therefore
\[
q(e)=q(vv)=q_v(v)=p(v)=1,
\]
a contradiction. Hence $v\notin T(q)$, and so $r\in T(q)$; that is, $q_r=p$.

We now prove that $q_s=p$. At the coordinates $e,s\in S$, we have
\[
q_s(e)=q(s)=0=p(e)
\qquad\text{and}\qquad
q_s(s)=q(s^2)=q(e)=0=p(s).
\]
At the coordinate $r$, since $vr=rs$ and $q_r=p$,
\[
q_s(r)=q(rs)=q(vr)=q_r(v)=p(v)=1=p(r).
\]
Finally, since $vs=r^2$,
\[
q_s(v)=q(vs)=q(r^2)=q_r(r)=p(r)=1=p(v).
\]
Therefore $q_s=p$, which contradicts $q_s\neq p$. This contradiction shows that $\tau$ is idempotent.
\end{proof}

For $n=3$, the cellular automaton $\tau$ is still idempotent, but $T=\{r\}$ is not compatible: $r^{-1}s=r^2s$ lies in $Sr\cap(S\setminus T)$, and \eqref{eq:C2} would require $p(r^2sr^{-1})=p(s)=0$ to equal $p(r^2s)=1$. For $n=2$ the construction degenerates, since $r^2s=s$.

\begin{remark}\label{rem:mechanism}
The proof of Proposition~\ref{prop:dihedral-family} shows exactly where hypothesis~\eqref{eq:H} is needed. In the proof of Theorem~\ref{thm:main}, for $s\in S\setminus T$ with $p(s)=p(e)$, hypothesis \eqref{eq:H} gives $ts = vt$ with $v \in S_{p(e)}$, so the copy of $p$ on $St$ forces $q(ts) = p(v) = p(e) \neq p(t)$. Over $D_n$, with $t=r$, we have instead $rs=vr$, so the copy of $p$ on $Sr$ forces $q(rs)=p(v)=1=p(r)$. Together with the values forced at the other coordinates, this produces a complete copy of $p$ on $Ss$.
\end{remark}

\section{Conclusions and open questions}\label{sec:conclusions}

We introduced compatible sets as a way of deciding idempotence of a lazy cellular
automaton directly from its unique active transition $p \in A^S$ and its writing
symbol $a$. Over every group, the existence of a compatible set is necessary for
non-idempotence (Proposition~\ref{prop:necessity}), and it is also sufficient when
every element of $S_a$ normalizes $S_{p(e)}$ (Theorem~\ref{thm:main}); in
particular, this settles the case of abelian groups. Moreover, a compatible set
exists if and only if one of the sets $C(t)$, $t \in S_a$, is compatible
(Proposition~\ref{prop:closure}), so under the hypothesis of Theorem \ref{thm:main}, idempotence can
be decided with a number of lookups that is polynomial in $|S|$ (Corollary~\ref{cor:poly}), in contrast with the exponential search over $A^{SS}$ given by Lemma~\ref{lem:q}.

All the conditions considered in this paper involve only products of elements of
$S$, so they depend only on the subgroup generated by $S$. Thus
Theorem~\ref{thm:main} applies whenever the elements of $S$ commute pairwise, and
the construction of Section~\ref{sec:dihedral-counterexamples} yields an idempotent lazy cellular
automaton admitting a compatible set over every group that contains elements $r$
and $s$ with $s^2 = e$, $srs = r^{-1}$ and $r$ of order at least $4$, possibly
infinite.

We propose the following open questions:
\begin{enumerate}
\item Is there a criterion for the idempotence of lazy cellular automata over arbitrary
groups that can be tested with a number of lookups polynomial in $|S|$?

\item For which groups $G$ does the existence of a compatible set characterize
non-idempotence for every lazy cellular automaton over $G$?

\item For which groups $G$ and finite subsets $S\subseteq G$ does every lazy cellular automaton with neighborhood $S$ that admits a compatible set also admit a compatible singleton?
\end{enumerate}

The class of groups described in Question 2 contains all abelian groups and excludes the dihedral groups $D_n$ with $n \geq 4$. However, an exhaustive computer search shows that, when $|A| = 2$, every lazy cellular automaton over the symmetric group $S_3$ or over the quaternion group $Q_8$ that admits a compatible set is non-idempotent, so $S_3$ and $Q_8$ are also in this class of groups. 

Related to Question 3, Example~\ref{ex:singletons} gives a neighborhood $S \subseteq \mathbb{Z}$ and a lazy cellular automaton that admits a compatible set but no compatible singleton.


\section*{Acknowledgments} 
The first and third authors were supported by SECIHTI \textit{Becas Nacionales para Estudios de Posgrado.} The fourth author was supported by a SECIHTI Postdoctoral Fellowship \emph{Estancias Posdoctorales por M\'exico}, No. I1200/320/2022.  

\section*{Declaration on the use of artificial intelligence}

The authors used ChatGPT 6 Astra and Claude Opus 5.5 to assist with mathematical exploration and discovery. The authors wrote and verified all arguments, and take
full responsibility for the manuscript.

\end{document}